\documentclass[a4paper,11pt]{amsart}
\usepackage{amsmath,amssymb,amsfonts}
\usepackage{epsfig}
\usepackage{mathrsfs}
\newtheorem{defn}{Definition}[section]
\newtheorem{thm}{Theorem}[section]
\newtheorem{prop}{Proposition}[section]
\newtheorem{cor}{Corollary}[section]
\newtheorem{rmk}{Remark}[section]
\newtheorem{lma}{Lemma}[section]

\def\N{{\rm I\kern-0.16em N}}
\def\R{{\rm I\kern-0.16em R}}
\def\E{{\rm I\kern-0.16em E}}
\def\P{{\rm I\kern-0.16em P}}
\def\F{{\rm I\kern-0.16em F}}
\def\B{{\rm I\kern-0.16em B}}
\def\C{{\rm I\kern-0.46em C}}
\def\G{{\rm I\kern-0.50em G}}
\newcommand{\ud}{\mathrm{d}}
\numberwithin{equation}{section}
\font\eka=cmex10
\usepackage{ae}
\def\ind{\mathrel{\hbox{\rlap{%
\hbox to 7.5pt{\hrulefill}}\raise6.6pt\hbox{\eka\char'167}}}}
\begin{document}
\title[Rate of convergence]
{Rate of convergence for discretization of integrals with respect to Fractional Brownian motion}
\author[Azmoodeh and Viitasaari]{Ehsan Azmoodeh \and Lauri Viitasaari}
\address{Department of Mathematics and System Analysis, Helsinki University of Technology\\
P.O. Box 1100, FIN-02015 TKK,  FINLAND} 

\begin{abstract}
In this article, an uniform discretization of stochastic integrals $\int_{0}^{1} f'_-(B_t)\ud B_t$, with respect to fractional Brownian motion with Hurst parameter 
$H \in (\frac{1}{2},1)$, for a large class of convex functions $f$ is considered. In $\big[$\cite{a-m-v}, Statistics \& Decisions, \textbf{27}, 129-143$\big]$, for any 
convex function $f$, the almost sure convergence of uniform discretization to such stochastic integral is proved. Here we prove $L^r$- convergence of uniform 
discretization to stochastic integral. In addition, we obtain a rate of convergence. It turns out that the rate of convergence can be brought as closely as possible to $H - \frac{1}{2}$.

\medskip


\smallskip

\end{abstract}

\maketitle

\section{Introduction}

It is well-known that fractional Brownian motion $B=\{ B_t\}_{t \in [0,1]}$, with Hurst parameter $H \neq \frac{1}{2}$ is neither a semimartingale nor a Markov process.
Therefore, according to \textit{Bichteler-Dellacherie} theorem, the classical Ito stochastic integration theory cannot be used to define a stochastic integral with 
respect to fractional Brownian motion. In last decade, many authors studied different possible ways to define stochastic integrals with respect to fractional Brownian 
motion. Essentially two different types of integrals can be defined:
\begin{itemize}
 \item The pathwise Riemann-Stieltjes integral $\int_{0}^{1} u_t \ud B_t$ exists if the integrand stochastic process $u = \{ u_t \}_{t \in [0,1]}$ has H\"older 
continuous sample paths of order $\alpha > 1-H$, as a result of Young integration theory \cite{y}. Z\"ahle \cite{z} extended this integral using the fractional integration by part 
formula, to some stochastic processes having some fractional smoothness.
\item The Skorokhod integral $($ or divergence integral $)$with respect to fractional Brownian motion. This integral is defined as adjoint operator of the Malliavin 
derivative. It is known that for enough regular stochastic processes, the difference of this integral with corresponding pathwise Riemann-Stieltjes integral can be 
explained with Malliavin trace operator. For more details see \cite{m}.
\end{itemize}

In \cite{b-c-f}, the authors studied the problem of finding a discrete approximation of the stochastic integral with respect to fractional Brownian motion defined
as divergence. They provide a discrete approximation of stochastic integrals of divergence type by means of the resolutions of the Fock space associated to fractional 
Brownian motion $B$.\\

The aim of this paper is to study in more details convergence of the uniform discretization of pathwise stochastic integrals
\begin{equation}\label{eq:into}
S= \int_{0}^{1} f'_-(B_t)\ud B_t,
\end{equation}
where $f: \R \to \R$ is a convex function. In \cite{a-m-v}, it is shown that such integrals can be understood in the generalized Lebesgue-Stieltjes integral sense.
Moreover, the authors considered the uniform discretization
\begin{equation*}
S_n := \sum_{i=1}^n f'_-(B_{\frac{i-1}{n}})(B_{\frac{i}{n}}-B_{\frac{i-1}{n}}), \quad n \in \N.
\end{equation*}
 They proved that $S_n$ converges to stochastic integral $S$ almost surely as $n$ tends to infinity. Note that, 
from financial application point of view, it is convenient to take the left hand points $\frac{i-1}{n}$ in the definition of $S_n$. In stochastic finance the 
discretization $S_n$ can be interpreted as total losses or gains of the discretized delta hedging strategy $($see \cite{a-m-v}$)$.\\

In this paper, we show that with some fine and detailed analysis one can prove $L^r$- convergence too, i.e.
\begin{equation*}
S_n \longrightarrow S \quad \text{ in } \ L^r,
\end{equation*}
as $n$ tends to infinity for some range of $r\ge 1$ and a large class of convex functions. To obtain such result, we use integration theory known as the generalized 
Lebesgue-Stieltjes integration theory, introduced by Z\"ahle \cite{z}, and developed by Nualart-Rascanu in \cite{n-r} together with Lemma \ref{l:estimate}. Moreover, 
we obtain a rate of convergence:
\begin{equation*}
\Big( \E \big| S_n - S \big|^r \Big)^{\frac{1}{r}} \leq C \left(\frac{1}{n}\right)^{ H - \frac{1}{2} - \epsilon} \quad n\ge 1,
\end{equation*}
for sufficiently small $\epsilon$ and $C$ is a constant independent of $n$.\\

The paper is organized as follows. In the section $2$, we state our main result. Section $3$ contains all auxiliary facts which we need to prove our main 
result. The section $4$ is devoted to the proofs.

\section{Main result}
Throughout the paper, $B= \{ B_t\}_{t\in [0,1]}$ stands for a fractional Brownian motion on the interval 
$[0,1]$ with Hurst parameter $H \in (\frac{1}{2},1)$. Let $f:\R \to \R$ be a convex 
function, and denote by $\mu$ the positive Radon measure corresponding to its second derivative. Let $p$ be a positive number such that
\begin{equation}
\label{range_p}
2H<p<\frac{H}{1-H}.
\end{equation} 
Define a function $C:\R\rightarrow\R$ by
\begin{equation}
\label{constant}
C(a)=\max(1,|a|)e^{-\frac{\min \{ a^2,(a-1)^2 \} }{2}}.
\end{equation} 
We consider the following technical assumptions related to the measure $\mu$:\\

$\mathbf{(H_{1})}$ \ For $p$ and $C(a)$ given in \ref{range_p} and \ref{constant}, it holds
\begin{equation*}
\int_{\R} C(a)^{\frac{1}{p}}\mu(\ud a) < \infty.
\end{equation*}

$\mathbf{(H_{2})}$ \ For $p$ and $C(a)$ given in \ref{range_p} and \ref{constant}, it holds

\begin{equation*}
\int_0^\infty C(\log a)^{\frac{1}{p}}\mu(\ud a) < \infty.
\end{equation*}
In what follows, the stochastic integrals with respect to fractional Brownian motion are 
understood in the sense of the generalized Lebesgue-Stieltjes integrals $($see section $3$ for more details$)$. Now we can state our main theorem.

\begin{thm}
\label{main_theorem}
Let $f:\R \to \R$ be a convex function. Put
\begin{equation*}
\label{RS_sum}
S_n = \sum_{i=1}^n f'_-(B_{\frac{i-1}{n}})(B_{\frac{i}{n}}-B_{\frac{i-1}{n}}), \ \text{ and } \ S = \int_0^1 f'_-(B_u)\ud B_u.
\end{equation*}
Let $r\in[1,p)$, for $p$ given in (\ref{range_p}). If the assumption $\mathbf{(H_1)}$ holds, then \\
\begin{description}
 \item[\textbf{(a)}] We have the convergence $S_n \rightarrow S$ in $L^r$.
\item[\textbf{(b)}] For every positive number $\beta$ satisfying
\begin{equation}
\label{range_beta}
1-H<\beta<\frac{H}{p},\quad \beta \ne 1-\frac{2H}{p},
\end{equation}
there exists a constant $C=C(H,\beta,r,p,f)$ such that
\begin{equation*}
\label{rate_of_convergence}
\Vert S_n-S \Vert_r = \big( \E |S_n - S|^r \big)^{\frac{1}{r}} \leq C \left(\frac{1}{n}\right)^{\frac{H}{p}-\beta}, \quad n\ge 1.
\end{equation*}
\end{description}
\end{thm}

\begin{rmk}
 For  $r\in[1,p)$, the assumption $\mathbf{(H_1)}$ implies that $S,S_n \in L^r$.
\end{rmk}

\begin{cor}\label{main_corollary}
Let $X= \{ X_t\}_{t \in [0,1]}$ be a geometric fractional Brownian motion, i.e. $X_t =e^{B_t}$. Put 
\begin{equation*}
\label{RS_sums2}
\tilde{S}_n = \sum_{i=1}^n f'_-(X_{\frac{i-1}{n}})(X_{\frac{i}{n}}-X_{\frac{i-1}{n}}), \ \text{ and } \ \tilde{S} = \int_0^1 f'_-(X_u)\ud X_u.
\end{equation*}

Let $p$ and $\beta$ be positive numbers such that (\ref{range_p}) 
and (\ref{range_beta}) holds and let $r\in[1,p)$. If the assumption $\mathbf{(H_2)}$ holds, then \\
\begin{description}
 \item[\textbf{(a)}] We have the convergence 
$\tilde{S}_n \rightarrow \tilde{S}$ in $L^r$.
\item[\textbf{(b)}]
There exists a constant $\tilde{C}=\tilde{C}(H,\beta,r,p,f)$ such that
\begin{equation*}
\label{rate_of_convergence2}
\Vert\tilde{S}_n - \tilde{S} \Vert_r= \big( \E |\tilde{S}_n - \tilde{S} |^r \big)^{\frac{1}{r}} \leq \tilde{C} \left(\frac{1}{n}\right)^{\frac{H}{p}-\beta}, \quad n\ge 1.
\end{equation*}
\end{description}
\end{cor}
\begin{rmk}
 For $r\in[1,p)$, the assumption $\mathbf{(H_2)}$ implies that $\tilde{S},\tilde{S_n} \in L^r$.
\end{rmk}

\begin{rmk}
According to Theorem 2.1, rate of convergence can be any number in a certain interval. Therefore, one can improve the rate by choosing suitable parameters 
$p$ and $\beta$. In the best, rate of convergence can be brought as closely as possible to $H - \frac{1}{2}$, by letting $p$ very close to $2H$ and $\beta$ very 
close to $1-H$. However, as a price the constants $C$ and $\tilde{C}$ become larger and tends to infinity as $p$ decreases to $2H$ or $\beta$ decreases to $1-H$. 
\end{rmk}
\begin{rmk} It is not clear whether the approximate rate $H-\frac{1}{2}$ is the best possible rate that one can obtain.

\end{rmk}

\subsection{Related results and comparison with Brownian motion}
For more smooth integrands one can use \textit{Young-Loeve} estimate to get a better rate. In following theorem, the stochastic integral coincides with the 
Riemann-Stieltjes integral $($see remark \ref{rmk:coinside}$)$.

\begin{thm}\label{main-thm-2}
Let $f:\R \to \R$ be a Lipschitz function. Put
\begin{equation*}
\hat{S}_n = \sum_{i=1}^n f(B_{\frac{i-1}{n}})(B_{\frac{i}{n}}-B_{\frac{i-1}{n}}), \ \text{ and } \ \hat{S} = \int_0^1 f(B_u)\ud B_u.
\end{equation*}
Let $r  \ge 1$. Then for every $\epsilon \in (0,2H-1)$, there exists a constant $\hat{C}=\hat{C}(\epsilon,H,r,f)$ such that \\
\begin{equation*}
\Vert\hat{S}_n - \hat{S} \Vert_r= \big( \E |\hat{S}_n - \hat{S} |^r \big)^{\frac{1}{r}} \leq \hat{C} \left(\frac{1}{n}\right)^{2H-1-\epsilon}, \quad n\ge 1.
\end{equation*}
\end{thm}

For standard Brownian motion $W=\{W_t\}_{t\in[0,1]}$ i.e. $H=\frac{1}{2}$, the stochastic integral is understood in the sense of It\'{o} integral. In this case, we 
consider the following assumption related to the measure $\mu$:\\

$\mathbf{(H_{3})}$ \ It holds
\begin{equation*}
\int_{\R} e^{-\frac{\min \{ a^2,(a-1)^2 \} }{2}}\mu(\ud a) < \infty.
\end{equation*}

Then, we obtain the following result:

\begin{thm}
\label{main-thm-bm}
Let $W= \{ W_t\}_{t\in [0,1]}$ be a standard Brownian motion and $f:\R \to \R$ be a convex function. Put
\begin{equation*}
\label{RS_sum}
T_n = \sum_{i=1}^n f'_-(W_{\frac{i-1}{n}})(W_{\frac{i}{n}}-W_{\frac{i-1}{n}}), \ \text{ and } \ T = \int_0^1 f'_-(W_u)\ud W_u.
\end{equation*}
Let $r \in [ 1,2]$. If the assumption  $\mathbf{(H_3)}$ holds, then\\
\begin{description}
 \item[\textbf{(a)}] We have the convergence $T_n \rightarrow T$ in $L^r$.
\item[\textbf{(b)}] There exists a constant $C=C(r,f)$ such that
\begin{equation*}
\label{rate_of_convergence3}
\Vert T_n-T \Vert_r = \big( \E |T_n - T|^r \big)^{\frac{1}{r}} \leq C \left(\frac{1}{n}\right)^{\frac{1}{4}}, \quad n\ge 1.
\end{equation*}
\end{description}
\end{thm}

\begin{rmk}
The rate $n^{- \frac{1}{4}}$ obtained in Theorem \ref{main-thm-bm} is sharp for the class of convex functions satisfying the assumption $\mathbf{(H_{3})}$. For example 
for the convex function $f(x)=(x-a)^+$, there exists a constant $C=C(f)$ such that we have  
\begin{equation*}
  C \ n^{- \frac{1}{4}} \le \big( \E |T_n - T|^2 \big)^{\frac{1}{2}}, \quad n\ge 2,
\end{equation*}
where $T_n$ and $T$ are as in Theorem \ref{main-thm-bm}.
\end{rmk}

\begin{rmk}
To compare with the case of fractional Brownian motion, we see that the rate of convergence is better for fractional Brownian motion than for standard Brownian motion 
if $H>\frac{3}{4}$, and worse if $H<\frac{3}{4}$. It is known that $($see \cite{ch}$)$ the mixed Brownian-fractional Brownian motion $X= W + B$ is a semimartingale, if 
$H \in( \frac34,1)$, and for $H\in \left( \frac12 , \frac34 \right]$, $X$ is not a semimartingale with respect to its own filtration $\F ^X$. 
\end{rmk}
For smooth functions, using It\'{o} isometry one can easily get the rate $n^{-\frac{1}{2}}$. This is the subject of the next theorem. It is known that this is the 
best possible rate one can get for Lipschitz functions. See \cite[Remark 3, p.694]{w-w} and references therein.

\begin{thm}\label{main-thm-bm2}
Let $W= \{ W_t\}_{t\in [0,1]}$ be a standard Brownian motion and $f:\R \to \R$ be a Lipschitz function. Put
\begin{equation*}
\hat{T}_n = \sum_{i=1}^n f(W_{\frac{i-1}{n}})(W_{\frac{i}{n}}-W_{\frac{i-1}{n}}), \ \text{ and } \ \hat{T} = \int_0^1 f(W_u)\ud W_u.
\end{equation*}
Let $r\geq 1$. Then there exists a constant $\hat{C}=\hat{C}(f,r)$ such that \\
\begin{equation*}
\Vert\hat{T}_n - \hat{T} \Vert_r= \big( \E |\hat{T}_n - \hat{T} |^r \big)^{\frac{1}{r}} \leq \hat{C} \frac{1}{\sqrt{n}}, \quad n\ge 1.
\end{equation*}
\end{thm}

\section{Auxiliary facts}

 \subsection{Pathwise stochastic integration in fractional Besov-type spaces} \label{gls}
Since fractional Brownian motion in not a semimartingale,  hence the stochastic integral with 
respect to fractional Brownian motion $B$ must be defined. Using the smoothness of 
 the sample paths of $B$, when $H \in (\frac{1}{2},1)$, one can define the so-called \textit{generalized Lebesgue-Stieltjes integral}. 
We shall give some details of the construction of generalized Lebesgue-Stieltjes integrals in this section. For more information see \cite[Section 2.1.2]{m}.

\begin{defn}
Fix $ 0 <\beta < 1 $.\\

(i) Let $ W^{\beta}_1 =  W^{\beta}_1 ([0,T])$ be the space of real-valued measurable functions $ f :[0,T] \to \mathbb{R}$ such that

\begin{equation*}
\Vert f \Vert _{1,\beta} := \sup _{0 \le s < t \le T} \left( \frac{|f(t) - f(s)|}{(t-s)^\beta} + \int _{s}^{t} \frac{|f(u) - f(s) |}{(u-s)^ {1+\beta }} du \right) < \infty .
\end{equation*}

(ii) Let $W^{\beta}_2 =  W^{\beta}_2 ([0,T])$ be the space of real-valued measurable functions  $ f :[0,T] \to \mathbb{R}$ such that

\begin{equation*}
\Vert f \Vert _{2,\beta} := \int_{0}^{T} \frac{|f(s)|}{s^ \beta} ds + \int_{0}^{T}\int_{0}^{s} \frac{|f(u) - f(s) |}{(u-s)^ {1+\beta }} du ds < \infty .
\end{equation*}

\end{defn}

\begin{rmk}\label{r:rmk1}

The Besov spaces are closely related to the spaces of H{\"o}lder continuous functions. More precisely, for any $ 0< \epsilon < \beta \wedge (1- \beta)$,

\vskip0.25cm

\begin{center}
$C^{\beta + \epsilon}([0,T]) \subset W^{\beta}_{1} ([0,T])\subset C^{\beta - \epsilon}([0,T]) \quad \text{and} \quad C^{\beta + \epsilon}([0,T]) \subset W^{\beta}_{2} ([0,T]) $,
\end{center}

where $C^{\gamma }([0,T])$ denotes H\"older continuous functions of order $\gamma$.

\end{rmk}

Recall that the trajectories of $B$ belong to $ C^{\gamma }([0,T]) $ almost surely for any $T>0$ and any $0<\gamma < H$. This follows from the Kolmogorov continuity theorem. By remark \ref{r:rmk1}, we obtain that the trajectories of $B$ belong to $ W^{\beta}_1 ([0,T]) $ almost surely for any $T>0$ and any $0<\beta < H$.\\

Denote by $\Gamma$ the Gamma-function. Recall the left-sided Riemann-Liouville fractional integral operator $I^\beta _+$ of order $\beta > 0$:

$$
(I^\beta _{0+} f)(s) = \frac{1}{\Gamma (\beta)} \int _0^sf(u) (s-u)^{\beta -1} du .
$$

The corresponding right-sided fractional integral operator $I^\beta _- $ is defined by

$$
(I^\beta _{t-}f)(s) = \frac{1}{\Gamma (\beta) }\int _s^t f(u) (u-s)^{\beta -1} du .
$$

\begin{rmk}

If $ f \in  W^{\beta}_1 ([0,T])$, then its restriction to $[0,t] \subseteq [0,T]$ belongs to $I^{\beta}_{-}(L_\infty ([0,t]))$. Also, if  $ f \in  W^{\beta}_2 ([0,T])$, then its restriction to $[0,t] \subseteq [0,T]$ belongs to $I^{\beta}_{+}(L_1 ([0,t]))$, where $I^{\beta}_{-}(L_\infty ([0,t])) $ (resp. $I^{\beta}_{+}(L_1 ([0,t])) $) stand for the image of $ L_\infty ([0,t])$ (resp. $ L_1 ([0,t])$) by the fractional Riemann-Liouville operator $I^{\beta}_{-} $ (resp. $ I^{\beta}_{+}$).(For details we refer to \cite{s-k-m}).

\end{rmk}

\begin{defn}

Let $f:[0,T] \to \mathbb{R}$ and $0< \beta < 1$. If $ f \in I^{\beta}_{+}(L_{1} ([0,T]))$(resp. $f \in I^{\beta}_{-}(L_\infty ([0,T]))$ then the Riemann-Liouville fractional derivatives are defined using the Weyl representation as

\begin{multline*}
(D^{\beta}_{0+} f)(x)= \frac{1}{\Gamma(1-\beta)} \left( \frac{f(x)}{x^\beta} + \beta \int_{0}^{x}\frac{f(x) - f(y)}{(x-y)^{\beta + 1}}dy \right) \textbf{1} _{(0,T)}(x),\\
\left( \text{resp}.(D^{\beta}_{T^{-}} f)(x)= \frac{1}{\Gamma(1-\beta)} \left( \frac{f(x)}{(T-x)^\beta} + \beta \int_{x}^{T}\frac{f(x) - f(y)}{(y-x)^{\beta + 1}}dy \right) \textbf{1} _{(0,T)}(x)\right).
\end{multline*}

\end{defn}

For a detailed discussion, we refer to \cite{s-k-m}. The following proposition clarifies the construction of the stochastic integrals. This approach is by Nualart and  Rascanu.

\begin{prop}\cite{n-r}\label{pr:n-r}
Let $ f \in  W^{\beta}_2 ([0,T])$, $ g \in W^{1- \beta}_1 ([0,T])$. Then for any $t \in(0,T]$ the Lebesgue integral

\begin{center}
 $\int_{0}^{t} (D^{\beta}_{0+} f)(x) (D^{1- \beta}_{t-} g_{t-} )(x) dx$
\end{center}

exists, and we can define the \textit{generalized Lebesgue-Stieltjes integral} by

\begin{equation*}
\int_{0}^t f dg := \int_{0}^{t} (D^{\beta}_{0+} f)(x) (D^{1- \beta}_{t-} g_{t-} )(x) dx .
\end{equation*}

\end{prop}

\begin{rmk}\label{rmk:coinside}
It is shown in \cite{z} that if $f \in C^{\gamma }([0,T]) $ and $g \in C^{\mu }([0,T])$ with $ \gamma + \mu > 1$, then the integral $ \int_{0}^{T} f dg $ exists 
in the sense of the generalized Lebesgue-Stieltjes integral and coincides with the Riemann-Stieltjes integral.
\end{rmk}

The next theorem can be used to study the continuity of the integral.

\begin{thm}\cite{n-r}\label{t:n-r}
Let $ f \in  W^{\beta}_2 [0,T]$ and  $ g \in W^{1- \beta}_1 [0,T]$. Then we have the estimation
\begin{equation}
\left|\int_{0}^t f dg \right| \le \sup_{0\le s < t \le 1} \big| D^{1 - \beta}_{t^{-}} g_{t^{-}}(s) \big| \Vert f \Vert _{2,\beta} ,\quad t\in[0,T].
\end{equation}
\end{thm}

Now we can state the existence of stochastic integral with respect to fractional Brownian motion in our main result.

\begin{thm}\cite{a-m-v}\label{thm:integral}
Let $f:\R \to \R $ be any convex function.\\

$(i)$ The stochastic integral
\begin{equation*}\label{eq:1}
\int _{0}^{1} f^{'}_{-}(B_t) dB_t
\end{equation*}
can be understood a.s. in the sense of the generalized Lebesgue-Stieltjes integral.\\

$(ii)$ The following Ito formula
\begin{equation*} \label{eq:Ito}
 f(B_1)= f(0) + \int_{0}^{1}f^{'}_{-}(B_t) dB_t
\end{equation*}
holds, where the stochastic integral is understood in the sense of the generalized Lebesgue-Stieltjes integral.\\

$(iii)$ One can approximate the stochastic integral by Riemann-Stieltjes sums. More precisely,
\begin{equation*}
\sum_{i=1}^{n} f^{'}_{-}  (B_{t^{n}_{i-1}}) (B_{t^{n}_i} - B_{t^{n}_{i-1}})     \stackrel{\text{a.s.}}{\longrightarrow} \int_{0}^{1} f^{'}_{-}  (B_{t}) dB_{t}, \quad t^{n}_i=\frac{i}{n}.
\end{equation*}
\end{thm}

\subsection{ Some results related to fractional Brownian motion}

The so-called Garsia-Rademich-Rumsey inequality provides basic inequalities on increments of continuous stochastic processes. Using this inequality, one can obtain 
the following lemma on the moments of supremum of fractional derivative of fractional Brownian motion.

\begin{lma}\cite{n-r}\label{l:n-r}
Let $B = \{ B_t \}_{t \in [0,1]}$ be a fractional Brownian motion with Hurst parameter $H \in (\frac{1}{2},1)$. Let $1-H < \beta < \frac{1}{2}$ and $p\ge 1$, then
\begin{equation*}
\E \big( \sup_{0\le s < t \le 1} \big| D^{1 - \beta}_{t^{-}} B_{t^{-}}(s) \big| \ \big)^p < \infty.
\end{equation*}
\end{lma}

We continue with an useful estimate of a probability that fractional Brownian motion crosses a fixed level. It turns out that this is a main 
ingredient for the proof of the main theorem. Actually the following result is an improvement of the Lemma 4 $($see \cite{c-n-t}$)$ with a better constant in terms of the 
level $a$.

\begin{lma}\label{l:estimate}
Let $B = \{ B_t \}_{t \in [0,1]}$ be a fractional Brownian motion with Hurst parameter $H \in \left(\frac{1}{2},1\right)$. Fix $0 < s < t \le 1$ and $ a \in \R$. 
Then there exists a constant $C$, independent of $s$, $t$ and $a$, such that the following 
estimate
\begin{equation*}
\P \big( B_t > a \ \text{and} \ B_s < a \big) \le C \  C(a) (t-s)^{H} s^{-2H}
\end{equation*}
holds.
\end{lma}
\begin{lma}\label{l:estimate_bm}
Let $W = \{ W_t \}_{t \in [0,1]}$ be a standard Brownian motion. Fix $0 < s < t \le 1$ and $ a \in \R$. Then there exists a constant $C$, independent of $s$, $t$ and 
$a$, such that the following 
estimate
\begin{equation*}
\P \big( W_t > a \ \text{and} \ W_s < a \big) \le C e^{-\frac{\min\{ a^2,(a-1)^2 \} }{2}} \sqrt{\frac{t-s}{s}}
\end{equation*}
holds.
\end{lma}
The proof of the lemmas are given in Appendix A. We also use the following well-known estimate for the tail probability of standard normal random variable.
\begin{lma}
\label{standard_estimate}
Let $Z$ be a standard normal random variable and fix $a>0$. Then
\begin{equation}
\P\big(Z>a\big) \leq \frac{1}{\sqrt{2\pi}a}e^{-\frac{a^2}{2}}.
\end{equation}
\end{lma}
\section{Proofs}
We start with the following simple lemma. It turns out that it provides enough good upper bound.

\begin{lma}
\label{lemma_fundamental}
Let $n\geq 2$ and $\alpha\in(0,1)$. Then
\begin{equation*}
\sum_{i=1}^{n-1}\left(\frac{1}{i}\right)^\alpha \leq \frac{1}{1-\alpha}n^{1-\alpha}. 
\end{equation*}
\end{lma}

\begin{proof}[Proof of theorem~\ref{main_theorem}]
Throughout the proof all constants will be denoted by $C$, and their values may differ from line to line. Random constants will be denoted by $C(\omega)$. We prove 
the statement only for $r=1$. The general case follows by similar arguments (see Remark \ref{remark_general_r}). Note that
$$
S_n -S=\int_0^1 h_n(t)\ud B_t
$$
where
\begin{equation}
\label{h_n_def}
h_n(t) = \sum_{i=1}^n \left(f'_-(B_{\frac{i-1}{i}}) - f'_-(B_t)\right) \textbf{1}_{\left(\frac{i-1}{n},\frac{i}{n}\right]}(t).
\end{equation}
By Theorem \ref{t:n-r} and Lemma \ref{l:n-r}, there exists a random variable $C(\omega,H,\beta)$ for which all the moments exists and
\begin{equation}
\label{Nualart_1}
|S_n-S| \leq C(\omega,H,\beta)\Vert h_n\Vert_{2,\beta}
\end{equation}
for every $\beta\in(1-H,\frac{1}{2})$.
Thus by H\"{o}lder inequality, we obtain
\begin{equation}
\label{Holder}
\E|S_n-S|\leq C(H,\beta,p) \left[\E\Vert h_n\Vert_{2,\beta}^p\right]^{\frac{1}{p}}.
\end{equation}
Let now $p$ be as in (\ref{range_p}) and let $\beta\in\left(1-H,\frac{H}{p}\right)$. We proceed to compute the term $\left[\E\Vert h_n\Vert_{2,\beta}^{p}\right]^{\frac{1}{p}}$. We have
$$
\left[\E\Vert h_n\Vert_{2,\beta}^{p}\right]^{\frac{1}{p}} \leq
\left(\E J_{n}^p\right)^{\frac{1}{p}} +\left(\E I_{n}^p\right)^{\frac{1}{p}}, 
$$
where $J_n$ denotes the first term and $I_n$ the second term in the Besov norm $\Vert\cdot\Vert_{2,\beta}$. The rest of the proof is split 
into three steps. We first prove the statement for a convex function $f(x)=(x-a)^+$, where $\ a \in \R$. Next we prove the statement for convex functions for which the measure
$\mu$ has compact support. Finally, we prove the result for convex functions for which the assumption $\mathbf{(H_1)}$ holds. 
\begin{enumerate}
\item[\textbf{Step 1.}]
The case $f(x)=(x-a)^+$.
\end{enumerate}
Now we have
\begin{equation}
\label{h_n_simple}
h^{a}_n(t) = \sum_{i=1}^n \left(\textbf{1}_{ \{ B_t< a <B_{\frac{i-1}{i}}\} }-\textbf{1}_{ \{ B_{\frac{i-1}{i}}< a <B_t \} }\right)\textbf{1}_{\left(\frac{i-1}{n},\frac{i}{n}\right]}(t).
\end{equation}

For the term $J_{n}$, we use Minkowski inequality for integrals to obtain 
\begin{equation*}
\begin{split}
\left(\E J_{n}^p\right)^{\frac{1}{p}} 
&\leq  \int_0^1 \frac{\left(\E|h^{a}_{n}(t)|^p\right)^{\frac{1}{p}}}{t^\beta}\ud t\\
&\leq  \int_0^{\frac{1}{n}}\frac{ \P (B_t> |a|) ^{1/p} }{t^\beta}\ud t \\
&+ \sum_{i=2}^n\int_{\frac{i-1}{n}}^{\frac{i}{n}}
\frac{\P (B_t>a>B_{\frac{i-1}{n}}) ^{1/p} +\P(B_t<a<B_{\frac{i-1}{n}}) ^{1/p}}{t^\beta}\ud t\\
& := J_{n,1}+J_{n,2}.
\end{split}
\end{equation*}

In $J_{n,1}$, the probability can be 
estimated by one, if $|a|\leq 1$, and by estimate $3.2$ if $|a|>1$. Hence the term $J_{n,1}$ can be bounded as
$$
J_{n,1} \leq C\int_0^{\frac{1}{n}}\frac{1}{t^\beta}\ud t
\leq  C\left(\frac{1}{n}\right)^{1-\beta}. 
$$

For the term $J_{n,2}$, by symmetric property of fractional Brownian motion, it is sufficient to consider only the event $\{B_t < a < B_{\frac{i-1}{n}}\}$. Therefore, $J_{n,2}$ can be bounded using Lemma \ref{l:estimate} as
\begin{equation*}
\begin{split}
J_{n,2} & \leq C \sum_{i=2}^n\int_{\frac{i-1}{n}}^{\frac{i}{n}}
\frac{\left(t-\frac{i-1}{n}\right)^{H/p}\left(\frac{i-1}{n}\right)^{-2H/p}}{t^\beta}\ud t\\
&\leq  C\sum_{i=2}^n\left(\frac{1}{n}\right)^{H/p}\left(\frac{i-1}{n}\right)^{-2H/p}\int_{\frac{i-1}{n}}^{\frac{i}{n}}
\frac{1}{t^\beta}\ud t\\
&\leq C\left(\frac{1}{n}\right)^{1-\frac{H}{p}-\beta}\sum_{i=2}^n \left(\frac{1}{i-1}\right)^{\frac{2H}{p}+\beta}.
\end{split}
\end{equation*}
Together with lemma \ref{lemma_fundamental}, this implies that
\begin{equation*}
J_{n,2} \le
\begin{cases}
 C \left(\frac{1}{n}\right)^{\frac{H}{p}} & \text{ if $\alpha \in (0,1)$},\\
C \left(\frac{1}{n}\right)^{1-\beta - \frac{H}{p}} & \text{ if $\alpha > 1$},
\end{cases}
\end{equation*}

where $\alpha= \frac{2H}{p} + \beta$.

We proceed to study the term $I_n$. We split the integral into several 
parts. Particularly, we consider the cases when $s$ and $t$ lie in the same interval and when they lie in different intervals. Note that when $s\in\left(\frac{j-1}{n},\frac{j}{n}\right]$ and $t\in\left(\frac{i-1}{n},\frac{i}{n}\right]$ with $i\neq j$, we have
\begin{equation*}
\begin{split}
&|h^{a}_n(t)-h^{a}_n(s)|\\
&=|\textbf{1}_{\{ B_t< a < B_{\frac{i-1}{n}} \} }-\textbf{1}_{ \{ B_t> a >B_{\frac{i-1}{n}} \} }
-\textbf{1}_{ \{ B_s< a < B_{\frac{j-1}{n}} \} }+\textbf{1}_{ \{ B_s> a > B_{\frac{j-1}{n}} \} }|\\
&\leq |\textbf{1}_{ \{ B_t< a < B_{\frac{i-1}{n}}\} } - \textbf{1}_{ \{B_s< a < B_{\frac{j-1}{n}} \} }|+ |\textbf{1}_{ \{ B_s>a>B_{\frac{j-1}{n}}\} } - \textbf{1}_{ \{ B_t>a>B_{\frac{i-1}{n}} \} }|\\
& := H_1(j,i) + H_2(j,i).
\end{split}
\end{equation*}

Using Minkowski inequality for integrals, we have
\begin{equation*}
\begin{split}
\left(\E I_{n}^p\right)^{\frac{1}{p}} & \leq  \int_0^1\int_0^t \frac{\left(\E|h^{a}_n(t)-h^{a}_n(s)|^p\right)^{\frac{1}{p}}}{(t-s)^{\beta+1}}\ud s\ud t\\
&\leq  \int_0^{\frac{1}{n}}\int_0^t\frac{\P(B_t> a >B_s)^{1/p}+\P(B_t< a <B_s)^{1/p}}{(t-s)^{\beta+1}}\ud s\ud t \\
&+ \sum_{i=2}^n\int_{\frac{i-1}{n}}^{\frac{i}{n}}\int_{\frac{i-1}{n}}^t
\frac{\P(B_t> a > B_s)^{1/p}+\P(B_t< a <B_s)^{1/p}}{(t-s)^{\beta+1}}\ud s\ud t \\
&+ \sum_{i=2}^n\sum_{j=1}^{i-1}\int_{\frac{i-1}{n}}^{\frac{i}{n}}\int_{\frac{j-1}{n}}^{\frac{j}{n}}\frac{\left(\E H_1^p(j,i)\right)^{\frac{1}{p}}+\left(\E H_2^p(j,i)\right)^{\frac{1}{p}}}{(t-s)^{\beta+1}}\ud s \ud t\\
& := I_{n,1} + I_{n,2} + I_{n,3}.
\end{split}
\end{equation*}

We start with $I_{n,3}$. Note that it is enough to consider only the term $H_1(j,i)$. The term $H_2(j,i)$ can be
treated similarly. We have

\begin{equation*}
\left(\E H_1^p(j,i)\right)^{\frac{1}{p}}
\leq 2 \P\left(B_s< a < B_{\frac{j-1}{n}}\right)^{1/p}
+ 2 \P\left( B_t< a < B_{\frac{i-1}{n}}\right)^{1/p}.
\end{equation*}

Hence the term $I_{n,3}$ can be bounded as
\begin{equation*}
\label{I_n3}
\begin{split}
I_{n,3}&\leq  C\sum_{i=2}^n\sum_{j=1}^{i-1}\int_{\frac{i-1}{n}}^{\frac{i}{n}}\int_{\frac{j-1}{n}}^{\frac{j}{n}}\frac{\P\left(B_s< a <B_{\frac{j-1}{n}}\right)^{1/p}
+ \P\left(B_t< a <B_{\frac{i-1}{n}}\right)^{1/p}}{(t-s)^{\beta+1}}\ud s \ud t\\
&= C\sum_{i=2}^n\sum_{j=1}^{i-1}\int_{\frac{i-1}{n}}^{\frac{i}{n}}\int_{\frac{j-1}{n}}^{\frac{j}{n}}\frac{\P\left(B_s< a <B_{\frac{j-1}{n}}\right)^{1/p}
}{(t-s)^{\beta+1}}\ud s \ud t\\
&+ C\sum_{i=2}^n\sum_{j=1}^{i-1}\int_{\frac{i-1}{n}}^{\frac{i}{n}}\int_{\frac{j-1}{n}}^{\frac{j}{n}}\frac{
\P\left(B_t< a <B_{\frac{i-1}{n}}\right)^{1/p}}{(t-s)^{\beta+1}}\ud s \ud t\\
& := I^{(1)}_{n,3} + I^{(2)}_{n,3}.
\end{split}
\end{equation*}

For the term $I^{(2)}_{n,3}$, by using Lemma \ref{l:estimate}, we obtain
\begin{equation*}
\begin{split}
I^{(2)}_{n,3}&\leq C\sum_{i=2}^n\int_{\frac{i-1}{n}}^{\frac{i}{n}}\int_0^{\frac{i-1}{n}}\frac{\left(t-\frac{i-1}{n}\right)^{\frac{H}{p}}\left(\frac{i-1}{n}\right)^{-\frac{2H}{p}}
}{(t-s)^{\beta+1}}\ud s \ud t\\
&\leq C\sum_{i=2}^n\int_{\frac{i-1}{n}}^{\frac{i}{n}}
\left(t-\frac{i-1}{n}\right)^{\frac{H}{p}}\left(\frac{i-1}{n}\right)^{-\frac{2H}{p}}\left(t-\frac{i-1}{n}\right)^{-\beta}\ud t\\
&= C \left(\frac{1}{n}\right)^{1-\beta-\frac{H}{p}}
\sum_{i=2}^n\left(\frac{1}{i-1}\right)^{\frac{2H}{p}}\\
&\leq C \left(\frac{1}{n}\right)^{\frac{H}{p}-\beta}
\end{split}
\end{equation*}

where for the last inequality, we have used Lemma \ref{lemma_fundamental}. Next we consider the term $I^{(1)}_{n,3}$. In this case, we have to study the case 
$j=1$ separately. Let $j=1$ in the term $I^{(1)}_{n,3}$. Then, by proceeding as for $J_{n,1}$, we have
\begin{equation*}
\begin{split}
&\hspace*{-3cm} \sum_{i=2}^n\int_{\frac{i-1}{n}}^{\frac{i}{n}}\int_0^{\frac{1}{n}}\frac{ \P\left(B_s<a<0\right)^{1/p}}{(t-s)^{\beta+1}}\ud s \ud t\\
& \leq  C\sum_{i=2}^n\int_{\frac{i-1}{n}}^{\frac{i}{n}}\int_0^{\frac{1}{n}}\frac{
1}{(t-s)^{\beta+1}}\ud s \ud t\\
&= C\int_{\frac{1}{n}}^1 \left[t^{-\beta}-\left(t-\frac{1}{n}\right)^{-\beta}\right]\ud t\\
&= C\left[\left(1-\frac{1}{n}\right)^{1-\beta} - 1 + \left(\frac{1}{n}\right)^{1-\beta}\right]\\
&\leq  C\left(\frac{1}{n}\right)^{1-\beta}.
\end{split}
\end{equation*}

If $j>1$, then by changing the order of two summations and tedious manipulation, one gets

\begin{equation*}
 \sum_{i=2}^n\sum_{j=2}^{i-1}\int_{\frac{i-1}{n}}^{\frac{i}{n}}\int_{\frac{j-1}{n}}^{\frac{j}{n}}\frac{
\P\left(B_t< a <B_{\frac{i-1}{n}}\right)^{1/p}}{(t-s)^{\beta+1}}\ud s \ud t \le C \left( \frac{1}{n} \right)^{\frac{H}{p} - \beta}.
\end{equation*}

It remains to estimate the terms $I_{n,1}$ and $I_{n,2}$. For $I_{n,1}$, by using Lemma \ref{l:estimate}, we have
\begin{equation*}
\begin{split}
I_{n,1}&\leq C\int_0^{\frac{1}{n}}\int_0^t \frac{\left(t-s\right)^{\frac{H}{p}}s^{-\frac{2H}{p}}}{(t-s)^{\beta+1}}\ud s \ud t\\
&= C\int_0^{\frac{1}{n}}t^{\frac{H}{p}-\beta-1}\int_0^ts^{-\frac{2H}{p}}\left(1-\frac{s}{t}\right)^{\frac{H}{p}-\beta-1}\ud s\ud t\\
&=C\int_0^{\frac{1}{n}}t^{\frac{H}{p}-\beta}\int_0^1
(tu)^{-\frac{2H}{p}}\left(1-u\right)^{\frac{H}{p}-\beta-1}\ud u\ud t\\
&= C \ B(1-\frac{2H}{p},\frac{H}{p}-\beta )\int_0^{\frac{1}{n}}t^{-\frac{H}{p}-\beta}\ud t\\
&= C \left(\frac{1}{n}\right)^{1-\frac{H}{p}-\beta}
\end{split}
\end{equation*}
where $B(x,y)$ denotes the complete Beta function. For the term $I_{n,2}$, we obtain

\begin{equation*}
\begin{split}
I_{n,2}&\leq C\sum_{i=2}^n\int_{\frac{i-1}{n}}^{\frac{i}{n}}
\int_{\frac{i-1}{n}}^t\frac{\left(t-s\right)^{\frac{H}{p}}s^{-\frac{2H}{p}}}{(t-s)^{\beta+1}}\ud s \ud t\\
&\leq  C\sum_{i=2}^n \left(\frac{i-1}{n}\right)^{-\frac{2H}{p}}
\int_{\frac{i-1}{n}}^{\frac{i}{n}}\left(t-\frac{i-1}{n}\right)^{\frac{H}{p}-\beta}\ud t\\
&=C\left(\frac{1}{n}\right)^{1-\frac{H}{p}-\beta}
\sum_{i=2}^n\left(\frac{1}{i-1}\right)^{\frac{2H}{p}}\\
&\leq  C\left(\frac{1}{n}\right)^{\frac{H}{p}-\beta},
\end{split}
\end{equation*}
where we have used Lemma \ref{lemma_fundamental}. Finally, by collecting estimates for $J_{n,1}$, $J_{n,2}$, 
$I_{n,1}$, $I_{n,2}$ and $I_{n,3}$ 
we obtain that for the convex function $f(x)=(x-a)^+$, there exists a constant $C=C(H,\beta,p)$ such that
\begin{equation}
\label{result_call}
\E|S_n-S|\leq C \ C(a)^{\frac{1}{p}}\left(\frac{1}{n}\right)^{\frac{H}{p}-\beta}.
\end{equation}

\begin{enumerate}
\item[\textbf{Step 2.}]
The case $\text{supp}(\mu)$ is compact.
\end{enumerate}

It is well-known that left derivative of the convex function $f$ has the following representation  
\begin{equation*}
f'_-(x)=\frac{1}{2}\int_\R \text{sgn}(x-a)\mu (\ud a)
\end{equation*}
up to constant. So
\begin{equation*}
\begin{split}
h_n(t) &= \frac{1}{2}\sum_{i=1}^n \left(f'_-(B_{\frac{i-1}{i}}) - f'_-(B_t)\right)\textbf{1}_{\left(\frac{i-1}{n},\frac{i}{n}\right]}(t)\\
&=\frac{1}{2}\sum_{i=1}^n \int_\R\left(\text{sgn}(B_{\frac{i-1}{i}}-a) - \text{sgn}(B_t-a)\right)\mu(\ud a)\textbf{1}_{\left(\frac{i-1}{n},\frac{i}{n}\right]}(t)\\
&=\frac{1}{2}\int_\R \sum_{i=1}^n \Big[ (\textbf{1}_{ \{ B_{\frac{i-1}{i}}>a \} } -\textbf{1}_{ \{ B_t>a \} } ) - (\textbf{1}_{ \{ B_{\frac{i-1}{i}}<a \} } - \textbf{1}_{ \{ B_t<a\} } ) \Big] \textbf{1}_{\left(\frac{i-1}{n},\frac{i}{n}\right]}(t)\mu(\ud a)\\
&= \int_\R h_n^a(t)\mu(\ud a).
\end{split}
\end{equation*}
From this observation, we obtain
\begin{equation}
\label{h_n_gen1}
|h_n(t)|\leq \int_\R|h_n^a(t)|\mu(\ud a)
\end{equation}
and
\begin{equation}
|h_n(t)-h_n(s)|\leq \int_\R|h_n^a(t)-h_n^a(s)|\mu(\ud a).
\end{equation}

Hence, using Tonelli's 
theorem, Minkowski inequality for integrals and inequality $(\ref{h_n_gen1})$ we obtain
\begin{equation*}
\left[\E\left(\int_0^1\frac{|h_n(t)|}{t^{\beta}}\ud t\right)^p\right]^{\frac{1}{p}} 
\leq \int_\R\int_0^1 \frac{\left(\E|h_n^a(t)|^p\right)^{\frac{1}{p}}}{t^\beta}\ud t\mu(\ud a),
\end{equation*}

and
\begin{equation*}
\begin{split}
&\hspace*{-1cm}\left[\E\left(\int_0^1\int_0^t\frac{|h_n(t)-h_n(s)|}{(t-s)^{\beta+1}}\ud s\ud t\right)^p\right]^{\frac{1}{p}}\\
&\leq 
\int_\R\int_0^1\int_0^t \frac{\left(\E|h_n^a(t)-h_n^a(s)|^p\right)^{\frac{1}{p}}}{(t-s)^{\beta+1}}\ud s\ud t\mu(\ud a).
\end{split}
\end{equation*}
Therefore, using step 1 we can conclude that there exists a constant $C = C(H,\beta,p)$ such that
\begin{equation*}
\E|S_n-S| \leq C \int_{\R} C(a)^{\frac{1}{p}}\mu(\ud a) \left(\frac{1}{n}\right)^{\frac{H}{p}-\beta},
\end{equation*}
where $C(a)$ is given by (\ref{constant}).

\begin{enumerate}
\item[\textbf{Step 3.}]
The general case.
\end{enumerate}

Now take any convex function $f$ which satisfies the assumption $\mathbf{(H_{1})}$. For any $k\in\N$, define the measurable set $\Omega_k$ by
\begin{equation}
\Omega_k = \{\omega : \sup_{0\leq t\leq 1}|B_t|\in [0,k]\}
\end{equation}
and auxiliary convex functions $f_k$ by
\begin{equation}
f_k(x)=\begin{cases}
f'_-(-k)x+(f(-k)+f'_-(-k)k),&x<-k\\
f(x),& x\in[-k,k]\\
f'_+(k)x+(f(k)-f'_+(k)k),&x>k.
\end{cases}
\end{equation}
Denote by $\mu_k$ the positive Radon measure associated to the second derivative of convex function $f_k$. Then $\mu_k$ has compact support contained in $[-k,k]$. 
Let $S^k$ and $S_n^k$ stand for the stochastic integral and the uniform discretization as in main theorem corresponds to the convex function $f_k$. Note that on the 
set $\Omega_k$, we have $S_n^k = S_n$ and $S^k = S$ almost surely. Hence, by monotone convergence theorem, we have
\begin{equation*}
\begin{split}
\E|S_n-S| &= \lim_{k\rightarrow\infty}\E|S_n-S|\textbf{1}_{\Omega_k}\\
&= \lim_{k\rightarrow\infty}\E|S_n^k-S^k|\textbf{1}_{\Omega_k}\\
&\leq \lim_{k\rightarrow\infty}\E|S_n^k-S^k|.
\end{split}
\end{equation*}
Applying step 2 and assumption $\mathbf{(H_{1})}$ completes the proof.
\end{proof}

\begin{rmk}
\label{remark_general_r}
The result for $1 < r < p $ follows with the same argument by choosing suitable parameters in H\"older inequality.
\end{rmk}

\begin{proof}[Proof of corollary \ref{main_corollary}]
We prove the result for the function $f(x)=(x-a)^+$ with some positive constant $a$. For negative $a$, we have 
$(X_t - a)^+ = X_t - a$ and the result is trivial. Moreover, the result for general convex function $f$ satisfying the assumption 
$\mathbf{(H_{2})}$ follows by same arguments as in the 
proof of Theorem \ref{main_theorem}. Put
\begin{equation}
\label{h_n_def2}
h_n^{X,a}(t) = \sum_{i=1}^n \left(\textbf{1}_{\{X_{\frac{i-1}{n}}>a\}} - \textbf{1}_{\left\{X_t>a\right\}}\right) X_t\textbf{1}_{\left(\frac{i-1}{n},\frac{i}{n}\right]}(t).
\end{equation}
Then it follows that
$$
\tilde{S}_n -\tilde{S}=\int_0^1 h_n^{X,a}(t)\ud B_t.
$$
A simple calculation gives us
\begin{equation*}
h_n^{X,a}(t) = X_th_n^{\log a}(t)
\end{equation*}
where $h_n^{a}(t)$ is given by \ref{h_n_simple}. Hence for the first term $J_n$, we obtain
\begin{equation*}
\int_0^1 \frac{|h_n^{X,a}(t)|}{t^\beta}\ud t \leq \overline{X} \int_0^1\frac{|h_n^{\log a}(t)|}{t^\beta}\ud t,
\end{equation*}
where $\overline{X}=\sup_{0\leq t\leq 1}X_t$. Moreover, all moments of $\overline{X}$ are finite 
$($see \cite{l}$)$. So we can replace $C(\omega,H,\beta)$
in the inequality (\ref{Nualart_1}) by a new random variable $\tilde{C}(\omega,H,\beta)=C(\omega,H,\beta)\overline{X}$. 
Hence the result follows by Step 1 of the proof of Theorem \ref{main_theorem}.
Next we consider the second term $I_n$. Note that
$$
|h_n^{X,a}(t) - h_n^{X,a}(s)| \leq |h_n^{\log a}(t)|| X_t-X_s|+|X_s|| h_n^{\log a}(t)-h_n^{\log a}(t)|.
$$
For the term $|X_s||h_n^{\log a}(t)-h_n^{\log a}(t)|$, we can proceed as for $J_n$ in the proof of Theorem 
\ref{main_theorem}. For the term $|h_n^{\log a}(t)||X_t-X_s|$, for any $\beta'\in(0,H-\beta)$, using the H\"{o}lder 
continuity property of sample paths of $X_t$, we obtain
\begin{equation*}
\begin{split}
&\hspace{-1.5cm}\int_0^1\int_0^t \frac{|h_n^{\log a}(t)||X_t-X_s|}{(t-s)^{1+\beta}}\ud s \ud t\\
&\leq C(\omega)\int_0^1 |h_n^{\log a}(t)|\int_0^t (t-s)^{H-\beta'-\beta-1}\ud s \ud t\\
&\leq  C(\omega)\int_0^1 |h_n^{\log a}(t)|t^{H-\beta'-\beta} \ud t\\
&\leq  C(\omega)\int_0^1 \frac{|h_n^{\log a}(t)|}{t^\beta}\ud t,
\end{split}
\end{equation*}
where $C(\omega)=C(\omega, H,\beta)$ is a positive random variable for which all moments are finite. 
Hence the result follows from Theorem \ref{main_theorem}.
\end{proof}

\begin{proof}[Proof of theorem~\ref{main-thm-2}]
Since $f$ is a Lipschitz function, there exists an universal constant $L>0$ such that
\begin{equation*}
\vert f(y) - f(x)\vert \le L \vert y-x \vert, \quad \forall x,y \in \R.
\end{equation*}
It is also known that sample paths of fractional Brownian motion are of bounded $p-$variation almost surely for any $p>\frac{1}{H}$. So, let $p,q>\frac{1}{H}$, and using 
Young-Loeve estimate $($see \cite{f-v}$)$, we have for every $0 < \epsilon < H$
\begin{equation*}
\begin{split}
 \vert \hat{S}_n - \hat{S} \vert & \le C \sum_{i=1}^{n} \Vert f(B)\Vert_{q-\text{var}[\frac{i-1}{n},\frac{i}{n}]} \Vert B\Vert_{ p-\text{var}[\frac{i-1}{n},\frac{i}{n}]}\\
& \le C(\omega) \sum_{i=1}^{n} \left( \frac{1}{n}\right)^{2H-\epsilon}\\
& \le C(\omega) \left( \frac{1}{n}\right)^{2H-1-\epsilon}
\end{split}
\end{equation*} 
for some positive random variable $C(\omega)=C(\omega,\epsilon,H,f)$ for which all the moments are finite. Now the 
claim follows.
\end{proof}

\begin{proof}[Proof of theorem \ref{main-thm-bm}]
The result follows by considering the convex function $f(x)=(x-a)^+$ and applying It\'{o} isometry and lemma \ref{l:estimate_bm}.
The general case follows by the same arguments as in the proof of Theorem \ref{main_theorem} together with the assumption $\mathbf{(H_{3})}$. 
The details are left to the reader.
\end{proof}

\appendix
\section{Proofs of lemmas \ref{l:estimate} and \ref{l:estimate_bm}}
We begin with the following simple lemma which we use in the proof.
\begin{lma}
\label{appendix_lemma}
Let $H>\frac{1}{2}$ and fix $0<s\leq t\leq 1$. Put
\begin{equation*}
R(t,s) = \frac{1}{2}\left[t^{2H}+s^{2H} - (t-s)^{2H}\right].
\end{equation*}
Then there exists a constant $C$ such that
\begin{equation*}
1-\frac{R(s,s)}{R(t,s)}\leq C (t-s)^H s^{-H}.
\end{equation*}
\end{lma}
\begin{proof}
Note that since $H>\frac{1}{2}$, we have $R(s,s)\leq R(t,s)$. Let now $t>2s$. Then
\begin{equation*}
\frac{(t-s)^H}{s^H}\geq 1 -\frac{R(s,s)}{R(t,s)}.
\end{equation*}
Hence it is sufficient to consider the case $s\leq t \leq 2s$. In this case we have
\begin{equation*}
\frac{R(t,s)}{R(s,s)}\leq \frac{R(2s,s)}{R(s,s)}=2^{2H-1}.
\end{equation*}
Hence we only have to prove that 
\begin{equation*}
1-\frac{R(s,s)}{R(t,s)}\leq C \frac{(t-s)^H}{s^{-H}}\frac{R(t,s)}{R(s,s)}.
\end{equation*}
By putting $k=\frac{t}{s}$ and dividing with $s^{5H}$, this is equivalent to
\begin{equation*}
\left[k^{2H}-1-(k-1)^{2H}\right]\leq C(k-1)^H\left[k^{2H} + 1 - (k-1)^{2H}\right]^2.
\end{equation*}
Now we have $k\in[1,2]$. Hence
\begin{equation*}
\begin{split}
&\hspace*{-3cm}\frac{\left[k^{2H}-1-(k-1)^{2H}\right]}{(k-1)^H\left[k^{2H} + 1 - (k-1)^{2H}\right]^2}\\
&\leq \frac{k^{2H}-1}{(k-1)^H}\leq \frac{k^{2H}-1}{k^H-1}\\
&= k^H+1 \leq 2^H+1.
\end{split}
\end{equation*}
This completes the proof.
\end{proof}
\begin{proof}[Proof of lemma \ref{l:estimate}]
Let $R(t,s)$ denotes the covariance function of fractional Brownian motion given by
\begin{equation*}
R(t,s) = \frac{1}{2}\left[t^{2H}+s^{2H} - (t-s)^{2H}\right].
\end{equation*}
We make use of decomposition
\begin{equation*}
B_t = \frac{R(t,s)}{R(s,s)}B_s + \sigma Y,
\end{equation*}
where $Y$ is $N(0,1)$ random variable independent of $B_s$ and 
\begin{equation*}
\sigma^2 = \frac{R(t,t)R(s,s)-R(t,s)^2}{R(s,s)}.
\end{equation*}
Assume that 
$$
\frac{R(s,s)}{R(t,s)}(a-1) < a.
$$ 
Then we obtain 
\begin{equation*}
\begin{split}
\P(B_t > a > B_s) &= \int_{-\infty}^a \P\left(Y\geq \frac{a-\frac{R(t,s)}{R(s,s)}x}{\sigma}\right)\frac{1}{\sqrt{2\pi}s^H}e^{-\frac{x^2}{2s^{2H}}}\ud x\\
&= \int_{-\infty}^{\frac{R(s,s)}{R(t,s)}(a-1)} \P\left(Y\geq \frac{a-\frac{R(t,s)}{R(s,s)}x}{\sigma}\right)\frac{1}{\sqrt{2\pi}s^H}e^{-\frac{x^2}{2s^{2H}}}\ud x\\
&+ \int_{\frac{R(s,s)}{R(t,s)}(a-1)}^{a} \P\left(Y\geq \frac{a-\frac{R(t,s)}{R(s,s)}x}{\sigma}\right)\frac{1}{\sqrt{2\pi}s^H}e^{-\frac{x^2}{2s^{2H}}}\ud x\\
&:= I_1 + I_2.
\end{split}
\end{equation*}
We begin with $I_1$. 
By lemma \ref{standard_estimate} we have
$$
\P\left(Y\geq \frac{a-\frac{R(t,s)}{R(s,s)}x}{\sigma}\right) \leq \frac{1}{\sqrt{2\pi}A(x)}e^{-\frac{A(x)^2}{2}},
$$
where $A(x) = \frac{a-\frac{R(t,s)}{R(s,s)}x}{\sigma}$. Hence
\begin{equation*}
\begin{split}
I_1 &\leq \int_{-\infty}^{\frac{R(s,s)}{R(t,s)}(a-1)} \frac{1}{\sqrt{2\pi}A(x)}e^{-\frac{A(x)^2}{2}}\frac{1}{\sqrt{2\pi}s^H}e^{-\frac{x^2}{2s^{2H}}}\ud x\\
&\leq \frac{\sigma}{s^H}e^{-\frac{a^2}{2}}\int_{-\infty}^{\frac{R(s,s)}{R(t,s)}(a-1)} \frac{1}{2\pi}e^{-\frac{A(x)^2}{2}-\frac{x^2}{2s^{2H}} + \frac{a^2}{2}}\ud x\\
&=\frac{R(s,s)}{R(t,s)}\frac{\sigma}{s^H}e^{-\frac{a^2}{2}}\int_{1}^{\infty} \frac{1}{2\pi}e^{-\frac{y^2}{2\sigma^2}-\frac{\left[\frac{R(s,s)}{R(t,s)}(a-y)\right]^2}{2s^{2H}} + \frac{a^2}{2}}\ud y
\end{split}
\end{equation*}
Note that $\sigma\leq (t-s)^H$ and $R(s,s)\leq R(t,s)$, so it remains to show that the integral is bounded by a constant independent of $s$, $t$ and $a$. It is easy 
to see that 
\begin{equation*}
\begin{split}
&-\frac{y^2}{2\sigma^2}-\frac{\left[\frac{R(s,s)}{R(t,s)}(a-y)\right]^2}{2s^{2H}} + \frac{a^2}{2} \\
&= -\frac{1}{2\sigma^2}\left[\left(y-a\frac{R(s,s)}{R(t,s)^2}\bar{\sigma}^2\right)^2 + a^2\left(\frac{R(s,s)}{R(t,s)^2}\bar{\sigma}^2 - \bar{\sigma}^2 - \frac{R(s,s)^2}{R(t,s)^4}\bar{\sigma}^4\right)\right],
\end{split}
\end{equation*}
where 
\begin{equation*}
\frac{1}{\bar{\sigma}^2}=\frac{1}{\sigma^2} + \frac{R(s,s)}{R(t,s)^2}.
\end{equation*}
Now 
\begin{equation*}
\frac{1}{\bar{\sigma}^2}\geq 1
\end{equation*}
and 
\begin{equation*}
\left(\frac{R(s,s)}{R(t,s)^2}\bar{\sigma}^2 - \bar{\sigma}^2 - \frac{R(s,s)^2}{R(t,s)^4}\bar{\sigma}^4\right) \geq 0.
\end{equation*}
Hence
\begin{equation*}
\begin{split}
&\int_{1}^{\infty} \frac{1}{2\pi}e^{-\frac{y^2}{2\sigma^2}-\frac{\left[\frac{R(s,s)}{R(t,s)}(a-y)\right]^2}{2s^{2H}} + \frac{a^2}{2}}\ud y\\
&\leq \int_{1}^{\infty} \frac{1}{2\pi}e^{-\frac{1}{2\sigma^2}\left(y-a\frac{R(s,s)}{R(t,s)^2}\bar{\sigma}^2\right)^2}\ud y\\
&\leq \frac{1}{\sqrt{2\pi}}.
\end{split}
\end{equation*}
Hence for $I_1$, there exists a constant $C$ such that 
\begin{equation*}
I_1 \leq C e^{-\frac{a^2}{2}} \ \frac{(t-s)^H}{s^H}.
\end{equation*}
We proceed to study the term $I_2$. Note that $\sigma^2\geq 0$. Hence
\begin{equation*}
\frac{R(s,s)}{R(t,s)^2}\geq \frac{1}{R(t,t)}\geq 1.
\end{equation*}
As a consequence, there exists a constant\footnote{We have $C=1$ except when $a\in[0,1]$. In this case $C=e^{\frac{1}{2}}$.} $C$ such that
\begin{equation*}
e^{-\frac{x^2}{2s^{2H}}}
\leq Ce^{-\frac{\min \{ a^2,(a-1)^2 \} }{2}}
\end{equation*}
for every $a$ and every $x\in\left[\frac{R(s,s)}{R(t,s)}(a-1),a\right]$.
Hence
\begin{equation*}
\begin{split}
I_2 &= \int_{\frac{R(s,s)}{R(t,s)}(a-1)}^{a} \int_{A(x)}^{\infty}\frac{1}{\sqrt{2\pi}}e^{-\frac{y^2}{2}}\ud y
\frac{1}{\sqrt{2\pi}s^H}e^{-\frac{x^2}{2s^{2H}}}\ud x\\
&\leq \frac{1}{\sqrt{2\pi}s^H}e^{-\frac{\min \{ a^2,(a-1)^2 \} }{2}}\int_{\frac{R(s,s)}{R(t,s)}(a-1)}^{a} \int_{A(x)}^{\infty}\frac{1}{\sqrt{2\pi}}e^{-\frac{y^2}{2}}
\frac{1}{\sqrt{2\pi}}\ud y\ud x.
\end{split}
\end{equation*}
By applying Tonelli's theorem, the integral can be written as
\begin{equation*}
\begin{split}
&\int_{\frac{R(s,s)}{R(t,s)}(a-1)}^{a} \int_{A(x)}^{\infty}\frac{1}{\sqrt{2\pi}}e^{-\frac{y^2}{2}}
\ud y\ud x\\
&= \int_{\left[1-\frac{R(t,s)}{R(s,s)}\right]\frac{a}{\sigma}}^{\frac{1}{\sigma}}\frac{1}{\sqrt{2\pi}}e^{-\frac{y^2}{2}}\left[a-\frac{R(s,s)}{R(t,s)}(a-\sigma y)\right]\ud y\\
&+ \int_{\frac{1}{\sigma}}^{\infty}\frac{1}{\sqrt{2\pi}}e^{-\frac{y^2}{2}}\left[a-\frac{R(s,s)}{R(t,s)}(a-1)\right]\ud y\\
&=: I_{2,1}+I_{2,2}.
\end{split}
\end{equation*}
For $I_{2,2}$, by lemma \ref{standard_estimate}, we obtain
\begin{equation*}
I_{2,2} \leq C\max(1,|a|)\sigma \leq 2\max(1,|a|)(t-s)^H.
\end{equation*}
For $I_{2,1}$, by applying lemma \ref{appendix_lemma}, we obtain 
\begin{equation*}
I_{2,1} \leq C \max(1,|a|) \frac{(t-s)^H}{s^H}.
\end{equation*}
Note that if 
\begin{equation*}
\frac{R(s,s)}{R(t,s)}(a-1) > a,
\end{equation*}
then we proceed as for $I_1$ and obtain the result. This completes the proof.
\end{proof}

\begin{proof}[Proof of lemma \ref{l:estimate_bm}]
The result follows by the same arguments as in the proof of lemma \ref{l:estimate} together with the fact that for standard Brownian motion we have $R(s,s)=R(t,s)$ for $s\le t$.
\end{proof}

\vskip0.5cm

\textbf{Acknowledgements}.\\
The authors thank Esko Valkeila for discussions and comments which improved the paper. Ehsan Azmoodeh thanks the Magnus Ehrnrooth foundation for financial support. 
Lauri Viitasaari thanks the Finnish Doctoral Programme in Stochastics and Statistics for financial support.

\end{document}